\documentclass [11pt]{amsart} 
\usepackage{a4}
\usepackage{paralist}
\usepackage{graphicx}
\usepackage{amssymb}
\usepackage{float}
\usepackage{amsmath}
\usepackage{psfrag}

\usepackage{amsmath,amsthm}

\theoremstyle{plain}
\newtheorem{thm}{Theorem}

\newtheorem{prop}[thm]{Proposition}
\newtheorem{cor}[thm]{Corollary}

\theoremstyle{definition}
\newtheorem{defn}[thm]{Definition}

\theoremstyle{remark}

\newtheorem*{histrem}{A Historical Remark}

\DeclareMathOperator{\core}{core}
\DeclareMathOperator{\comp}{\textbf{K}}

\title{Coset intersection graphs for groups}
\author{Jack Button, Maurice Chiodo, Mariano Zeron-Medina Laris}

\begin{document}

\begin{abstract}
Let $H, K$ be subgroups of $G$. We investigate the intersection properties of left and right cosets of these subgroups.
\end{abstract}

\maketitle

If $H$ and $K$ are subgroups of $G$, then $G$ can be partitioned as the disjoint union of all left cosets of $H$, as well as the disjoint union of all right cosets of $K$. But how do these two partitions of $G$ intersect each other?

\begin{defn}
Let $G$ be a group, and $H$ a subgroup of $G$. A \textit{left transversal} for $H$ in $G$ is a set $\{t_{\alpha}\}_{\alpha \in I}\subseteq G$ such that for each left coset $gH$ there is precisely one $\alpha \in I$ satisfying $t_{\alpha}H = gH$. A \textit{right transversal} for $H$ in $G$ in defined in an analogous fashion. A \emph{left-right} transversal for $H$ is a set $S$ which is simultaneously a left transversal, and a right transversal, for $H$ in $G$.
\end{defn}

A useful tool for studying the way left and right cosets interact, and obtaining transversals, is the coset intersection graph which we introduce here.

\begin{defn}
Let $G$ be a group and $H, K$ subgroups of $G$. We define the \emph{coset intersection graph} $\Gamma^{G}_{H,K}$ to be a graph with vertex set consisting of all left cosets of $H$ $(\{l_{i}H\}_{i \in I})$ together with all right cosets of $K$ $(\{Kr_{j}\}_{j \in J})$, where $I$, $J$ are index sets. If a left coset of $H$ and right coset of $K$ correspond, they are still included twice. Edges (undirected) are included whenever any two of these cosets intersect, and the edge $aH - Kb$ corresponds to the set $aH \cap Kb$.
\end{defn}

Observing that left (respectively, right) cosets do not intersect, we see that $\Gamma^{G}_{H,K}$ is a bipartite graph, split between $\{l_{i}H\}_{i \in I}$ and $\{Kr_{j}\}_{j \in J}$.

For $H$ a finite index subgroup of $G$, the existence of a left-right transversal is well known, sometimes presented as an application of Hall's marriage theorem \cite{Hall}. When $G$ is finite $H$ will have size $n$, so any set of $k$ left cosets of $H$ intersects at least $k$ right cosets of $H$ (or their union would have size $<kn$). Hence by Hall's theorem there is a matching on the bipartite graph $\Gamma^{G}_{H,H}$, and thus a left-right transversal (take one element from each edge in this matching). When $G$ is infinite the same argument applies to the finite quotient $G/\core(H)$ (The core of $H$, $\core(H)$, is the intersection of all conjugates of $H$ in $G$, $\bigcap_{g \in G} g^{-1}Hg$; it is always normal, and will be of finite index in $G$ whenever $H$ is).

The purpose of this paper is to show that in fact a much stronger result is true: we can completely describe the way that left and right cosets of $H$ intersect, without any need for Hall's theorem, but instead by studying and applying the properties of the coset intersection graph. We begin this now.

\begin{thm}\label{complete bipartite}
$\Gamma^{G}_{H,K}$ is always a disjoint union of complete bipartite graphs.
\end{thm}

\newpage

\begin{proof}
We first show that for $a,b,c,d \in G$ if 
$aH-Kb-cH-Kd$ is a path in $\Gamma^{G}_{H,K}$ then there is an edge $aH-Kd$. Note that there exist $h_{1}, h_{2}, h_{3} \in H$ and $k_{1}, k_{2}, k_{3} \in K$ such that $ah_{1}=k_{1}b$, $k_{2}b=ch_{2}$, $ch_{3}=k_{3}d$. Re-arranging gives $c=k_{3}dh_{3}^{-1}$, so $b=k_{2}^{-1}k_{3}dh_{3}^{-1}h_{2}$, so $a=k_{1}k_{2}^{-1}k_{3}dh_{3}^{-1}h_{2}h_{1}^{-1}$ and thus $ah_{1}h_{2}^{-1}h_{3}=k_{1}k_{2}^{-1}k_{3}d$. Hence $aH-Kd$ as required.
\\Now take any $l_{i}H$, and some $Kr_{j}$ in the connected component of $l_{i}H$ in $\Gamma^{G}_{H,K}$ (there is at least one such $Kr_{j}$); we show $l_{i}H$ and $Kr_{j}$ are connected by an edge. For if not, then there must be at least one finite path connecting them; take a minimal such path $\gamma$ from $l_{i}H$ to $Kr_{j}$. Then $\gamma$ begins with $l_{i}H-Ka-bH-Kc-\ldots$, where $Ka \neq Kr_{j}$. But by the previous remark, $l_{i}H$ and $Kc$ must be joined by an edge, contradicting the minimality of $\gamma$. So $l_{i}H$ and $Kr_{j}$ are joined by an edge, for every $Kr_{j}$ in the connected component of $l_{i}H$.
\end{proof}

Recall that $\comp_{s,t}$ denotes the complete bipartite graph on $(s,t)$ vertices. 
By imposing finiteness conditions on subgroups (finite index, or finite size), the graph $\Gamma^{G}_{H,K}$ exhibits an even greater level of symmetry. 

\begin{thm}\label{graph}
Let $H,K <G$. Suppose that either $|H|=m$, $|K|=n$ (where both subgroups are finite), or $|G:H|=n$,  $|G:K|=m$ (where both subgroups have finite index). Then the graph $\Gamma^{G}_{H,K}$ is a collection of disjoint, finite, complete bipartite graphs, where each component is of the form $\comp_{s_{i}, t_{i}}$ with $s_{i}/t_{i} = n/m$.
\end{thm}

\begin{proof}
Case 1: $|H|=m$, $|K|=n$. 
Take a connected component of $\Gamma^{G}_{H,K}$, which from theorem \ref{complete bipartite} must look like $\comp_{s,t}$ (as $|H|, |K|$ are finite) with vertices given by $s$ left cosets of $H$ and $t$ right cosets of $K$. Thus, in $G$, the disjoint union of these $s$ left cosets must be set-wise equal to the disjoint union of these $t$ right cosets. So $s|H|=t|K|$, and hence $s/t = n/m$.
\\Case 2: $|G:H|=n$,  $|G:K|=m$. 
Take $\core(H \cap K)$, which must be finite index in $G$ (say $|G: \core(H \cap K)|=l$), as $H, K$ and hence $H\cap K$ are. Now form the quotient $G/\core(H \cap K)$. Set $H':=H/\core(H \cap K)$, $K':= K/ \core(H \cap K)$. Since $|G:\core(H \cap K)|= |G: H|\cdot |H:\core(H \cap K)|$, we have that $|H'|=l/n$. Similarly, $|K'|=l/m$. Now apply case 1 to $G/\core(H \cap K)$, $H'$, $K'$.
\end{proof}

Under the hypotheses of the above theorem, we see that sets of $s_{i}$ left cosets of $H$ completely intersect  sets of $t_{i}$ right cosets of $K$, with $s_{i}/t_{i}$ constant over $i$. By drawing left cosets of $H$ as columns, and right cosets of $K$ as rows, we partition $G$ into irregular `chessboards' (denoted $C_{i}$) each with edge ratio $n:m$. Each chessboard $C_{i}$ corresponds to the connected component $\comp_{s_{i}, t_{i}}$ of $\Gamma^{G}_{H,K}$, and individual tiles in $C_{i}$ correspond to the nonempty intersection of a left coset of $H$ and a right coset of $K$ (i.e., edges in $\comp_{s_{i}, t_{i}}$). By choosing one element from each tile on a leading diagonal of the $C_{i}$'s (equivalently, one element from each edge in a maximum matching of the $\comp_{s_{i}, t_{i}}$'s), we deduce a stronger version of Hall's theorem for transversals:

\begin{cor}\label{useful}
Let $H,K<G$ be of finite index, with $|G:H|=m$ and $|G:K|=n$, where $m \leq n$. Then there exists a set $T \subseteq G$ which is a left transversal for $H$ in $G$, and which can be extended to a right transversal for $K$ in $G$. If $H=K$ in $G$, then $T$ becomes a left-right transversal for $H$.
\end{cor}

We now compute the sizes of the complete bipartite components of $\Gamma^{G}_{H,K}$.

\begin{prop}
Let $H,K<G$ and $g \in G$. Then the number of right cosets of $K$ intersecting $gH$ (call this $M_{g}$) satisfies:
\\$1$. $M_{g}=\frac {|G:gHg^{-1} \cap K|}{|G:H|}$ 
if $|G:H|, |G:K|$ are both finite.
\\$2$. $M_{g}=\frac{|H|}{|gHg^{-1} \cap K|}$ if $|H|,|K|$ are both finite.
\\A symmetric result applies for the number of left cosets of $H$ intersecting $Kg$.
\end{prop}

\begin{proof}
Let $N:= \core(H \cap K)$. We show that if $gH \cap Ka \neq \emptyset$ for some $a \in G$, then the number of cosets of $N$ in $gH \cap Ka $ is the same as the number in $gHg^{-1} \cap K $, independent of $a$ (in each of case 1 or 2 this number will be finite). So, as $gH \cap Ka \neq \emptyset$, we must have $gh=ka$ for some $h \in H$, $k \in K$. As $N$ is normal, we have the the number of cosets of $N$ in $gH \cap Ka $ is the same as the number in $gHa^{-1} \cap K = gHh^{-1}g^{-1}k \cap K=gHg^{-1}k \cap K$, which is the same as the number in $gHg^{-1} \cap Kk^{-1}=gHg^{-1} \cap K$ (observe that this number will be $|gHg^{-1} \cap K : N|$). 
This immediately gives:
\\(number of cosets of $N$ in $gH$) = (number of cosets of $N$ in $gHg^{-1} \cap K $)$\cdot M_{g}$ 
\\Thus $M_{g}=\frac {|H:N|}{|gHg^{-1} \cap K:N|}$. Both cases of the proposition now follow.
%
\end{proof}

All of our results can be derived from the work of Ore \cite{Ore}, who 
makes use of \emph{double cosets}; partitions of $G$ into sets of the form $KgH$ (where $H, K<G$). It follows that the complete bipartite components of $\Gamma^{G}_{H,K}$ from theorem \ref{complete bipartite} (the `chessboards') correspond to the double cosets of $G$; a left coset $aH$ and a right coset $Kb$ intersect if and only if they lie in the same double coset $KxH$. The symmetry exhibited by the coset intersection graph is not immediately obvious from Ore's use of terminology, and our exposition is more direct.

\begin{histrem}[with contributions from Warren Dicks and Jack Schmidt] The results in this paper have a somewhat piecemeal historical origin. A weaker version of corollary \ref{useful}, that a subgroup of a finite group always has a left-right transversal, appeared in 1910 by Miller \cite{Miller1}. In 1913 Chapman \cite{Chap1} proved the same result; he then realised the existence of the proof by Miller and in 1914 issued a corrigendum \cite{Chap2}. In 1927 Scorza \cite{Scorza} proved corollary \ref{useful} for two separate subgroups $H,K$ but still taking $G$ to be finite (the first time such a proof used double cosets). By the time of Zassenhaus' text \cite{Zas} in 1937, corollary \ref{useful} was known for finite index subgroups of infinite groups (the first time such a proof used Hall's theorem). In 1941 Sh\"{u} \cite{Shu} addressed this problem, in a way that leaves us somewhat confused. In 1958 Ore \cite{Ore} expanded significantly on such ideas, and gives what is to-date the most complete treatment of these, as well as his own historical account.
\end{histrem}

\

\begin{tt}
\noindent Selwyn College, University of Cambridge
\\Grange Road, Cambridge, CB3 9DQ, UK
\\J.O.Button@dpmms.cam.ac.uk
\\
\\Dipartimento di Matematica `Federigo Enriques'
\\Universit\`{a} degli Studi di Milano
\\Via Cesare Saldini 50, Milano, 20133, ITALIA
\\maurice.chiodo@unimi.it
\\
\\Department of Pure Mathematics and Mathematical Statistics
\\Centre for Mathematical Sciences, University of Cambridge
\\Wilberforce Road, Cambridge, CB3 0WB, UK
\\marianozeron@gmail.com

\end{tt}

\end{document}